\documentclass[12pt,reqno]{amsart}
\usepackage{amsfonts}
\usepackage{amssymb, latexsym}
\usepackage{eufrak}
\usepackage{lipsum} 
\setbox0=\hbox{$+$}
\newdimen\plusheight
\plusheight=\ht0
\def\+{\;\lower\plusheight\hbox{$+$}\;}

\setbox0=\hbox{$-$}
\newdimen\minusheight
\minusheight=\ht0
\def\-{\;\lower\minusheight\hbox{$-$}\;}

\setbox0=\hbox{$\cdots$}
\newdimen\cdotsheight
\cdotsheight=\plusheight
\def\cds{\lower\cdotsheight\hbox{$\cdots$}}

\renewcommand{\(}{\left\(}
\renewcommand{\)}{\right\)}
\renewcommand{\[}{\left[}

\numberwithin{equation}{section}
 \theoremstyle{plain}
\newtheorem{theorem}{Theorem}[section]
\newtheorem{lemma}[theorem]{Lemma}

\newtheorem{corollary}[theorem]{Corollary}

\newtheorem{remark}[theorem]{Remark}

\newenvironment{pf1}
   {\vskip 0.15in \par\noindent{\it Proof of Theorem \ref{th-a10n9}.}\hskip 0.5em\ignorespaces}

   \newenvironment{pf-a1-10}
   {\vskip 0.15in \par\noindent{\it Proof of \eqref{1,10mod5}.}\hskip 0.5em\ignorespaces}

    \newenvironment{pf-a3-10}
   {\vskip 0.15in \par\noindent{\it Proof of \eqref{3,10mod5}.}\hskip 0.5em\ignorespaces}

\begin{document}
\title[Proofs of Some Conjectures of Chan on Appell-Lerch Sums] {Proofs of Some Conjectures of Chan on Appell-Lerch Sums}

\author{Nayandeep Deka Baruah}
\address{Department of Mathematical Sciences, Tezpur University, Sonitpur, Assam, India, Pin-784028}
\email{nayan@tezu.ernet.in}

\author{Nilufar Mana Begum}
\address{Department of Mathematical Sciences, Tezpur University, Sonitpur, Assam, India, Pin-784028}
\email{nilufar@tezu.ernet.in}


\begin{center}
{\textbf{Proofs of Some Conjectures of Chan on Appell-Lerch Sums}}\\[5mm]
{\footnotesize  Nayandeep Deka Baruah and Nilufar Mana Begum}\\[3mm]
\end{center}

\vskip 5mm \noindent{\footnotesize{\bf Abstract.}
On page 3 of his lost notebook, Ramanujan defines the Appell-Lerch sum
$$\phi(q):=\sum_{n=0}^\infty \dfrac{(-q;q)_{2n}q^{n+1}}{(q;q^2)_{n+1}^2},$$ which is connected to some of his sixth order mock theta functions. Let $\sum_{n=1}^\infty a(n)q^n:=\phi(q)$. In this paper, we find a representation of the generating function of $a(10n+9)$ in terms of $q$-products. As corollaries, we deduce the congruences $a(50n+19)\equiv a(50n+39)\equiv a(50n+49)\equiv0~(\textup{mod}~25)$ as well as $a(1250n+250r+219)\equiv 0~(\textup{mod}~125)$, where $r=1$, $3$, and $4$. The first three congruences were  conjectured by Chan in 2012,  whereas the congruences modulo 125 are new. We also prove two more conjectural congruences of Chan for the coefficients of two Appell-Lerch sums.
\vskip 3mm
\noindent{\footnotesize Key Words:} Appell-Lerch sum;  Theta function; Mock theta function; Congruence

\vskip 3mm
\noindent {\footnotesize 2010 Mathematical Reviews Classification
Numbers: Primary 11P83; Secondary 33D15}.}

\section{\textbf{Introduction}}
Throughout the paper, we use the customary $q$-series notation:
\begin{align*}(a;q)_0&:=1,\\
(a;q)_n&:=\prod_{k=0}^{n-1}(1-aq^k), \quad n\ge 1,\\
(a;q)_\infty&:=\lim_{n\to \infty}(a;q)_n, \quad |q|<1,\\\intertext{and}
(a_1,a_2,\ldots,a_k;q)_\infty&:=(a_1;q)_\infty(a_2;q)_\infty\cdots(a_k;q)_\infty.\end{align*} For any positive integer $j$, for brevity, we also use $E_j:=(q^j;q^j)_\infty$.

Let $x, z\in\mathbb{C}^*$ with neither $z$ nor $xz$ an integral power of $q$. Following the definition given by Hickerson and Mortenson in \cite[Definition 1.1]{hickerson}, an Appell-Lerch sum $m(x,q,z)$ is a series of the form
$$m(x,q,z):=\dfrac{1}{(q,q/z,q;q)_\infty}\sum_{r=-\infty}^\infty\dfrac{(-1)^{n+1}q^{n(n+1)/2}z^{n+1}}{1-xzq^n}.$$ These sums were first studied in the nineteenth century by Appell \cite{appell1, appell2, appell3} and then by Lerch \cite{lerch}. But, in recent years, there has been considerable work on these sums and their connections to mock theta functions. We refer the readers to \cite{andrews-hick, chan-actarith, hickerson,  hickerson2,  mortenson1, mortenson2, waldherr, zwegers}.

In his lost notebook \cite[pp. 2,~4,~13,~17]{ramalost}, Ramanujan recorded seven mock theta functions and eleven identities involving them. Andrews and Hickerson \cite{andrews-hick} proved these eleven identities and called the seven functions sixth order mock theta functions. Three of the sixth order mock theta functions are\begin{align*}\rho(q)&:=\sum_{n=0}^\infty \dfrac{(-q;q)_{n}q^{n(n+1)/2}}{(q;q^2)_{n+1}},\\
\mu(q)&:=\sum_{n=0}^\infty \dfrac{(-1)^n(q;q^2)_nq^{(n+1)^2}}{(-q;q)_{2n+1}}\\\intertext{and}
\lambda(q)&:=\sum_{n=0}^\infty \dfrac{(-1)^n(q;q^2)_nq^{n}}{(-q;q)_{n}}.\end{align*}
On page 3 of his lost notebook \cite{ramalost}, Ramanujan defines the function
$$\phi(q):=\sum_{n=0}^\infty \dfrac{(-q;q)_{2n}q^{n+1}}{(q;q^2)_{n+1}^2},$$ and then states that
$$\rho(q)=2q^{-1}\phi(q^3)+\dfrac{(q^2;q^2)_\infty^2(-q^3;q^3)_\infty}{(q;q^2)_\infty^2(q^3;q^3)_\infty}.$$
Choi \cite{choi} proved two analogous identities involving $\phi$ and the two functions $\mu$ and $\lambda$. The function $\phi(q)$ was also studied by Hikami \cite{hikami}.

Now, let $\sum_{n=1}^\infty a(n)q^n:=\phi(q)$. Chan \cite{chan-actarith} proved several congruences for the coefficients $a(n)$ of the function $\phi$ modulo 2, 3, 4, 5, 7, and 27. In particular, Chan \cite{chan-actarith} proved the congruence
\begin{align}\label{a10n9}
a(10n+9)&\equiv0~(\textup{mod}~5)
\end{align}
and conjectured (\cite[Conjecture 7.1]{chan-actarith}) that, for any nonnegative integer $n$,
\begin{align}\label{a50n}
a(50n+19)&\equiv a(50n+39)\equiv a(50n+49)\equiv0~(\textup{mod}~25).
\end{align}

In this paper, we find a representation of the generating function of $a(10n+9)$ so that \eqref{a10n9} follows trivially. We then prove \eqref{a50n} from that representation of the generating function of $a(10n+9)$. Furthermore, we find the following new congruences:

\noindent For any nonnegative integer $n$, we have
\begin{align}\label{a1250n}
a(1250n+250r+219)&\equiv 0~(\textup{mod}~125), \quad \textup{for}~ r=1,3,4.
\end{align}

In \cite{chan-actarith}, Chan studied some other functions similar to $\phi$ and found congruences for them. In particular, he considered, for any integer $p\ge 2$ and $1\le j\le p-1$ with $p$ and $j$ coprime, the Appell-Lerch sum
$$\sum_{n=0}^\infty a_{j,p}(n)q^n=\dfrac{1}{(q^j,q^{p-j},q^p;q^p)_\infty}\sum_{n=-\infty}^\infty \dfrac{(-1)^nq^{pn(n+1)/2+jn+j}}{1-q^{pn+j}},$$ and proved that
$$\sum_{n=0}^\infty a_{j,p}(pn+(p-j)j)q^n=p~\dfrac{E_p^4}{E_1^3(q^j,q^{p-j};q^p)_\infty^2},$$ which readily implies the congruence
$$a_{j,p}(pn+(p-j)j)\equiv 0~(\textup{mod}~p).$$ It is to be noted that $2a(n)=a_{1,2}(n).$

In \cite{chan-mao},  Chan and Mao gave a generalization of $a_{j,p}$.

Chan \cite[Conjecture 7.1]{chan-actarith} also presented the following conjectural congruences:
\begin{align}
\label{1,6mod2}a_{1,6}(2n)&\equiv 0~(\textup{mod}~2),\\
\label{1,10mod2}a_{1,10}(2n)&\equiv a_{3,10}(2n)\equiv 0~(\textup{mod}~2),\\
\label{1,6mod3}a_{1,6}(6n+3)&\equiv 0~(\textup{mod}~3),\\
\label{1,3mod5}a_{1,3}(5n+3)&\equiv a_{1,3}(5n+4)\equiv 0~(\textup{mod}~5),\\
\label{1,10mod5}a_{1,10}(10n+5)&\equiv 0~(\textup{mod}~5),\\
\label{3,10mod5}a_{3,10}(10n+5)&\equiv 0~(\textup{mod}~5).
\end{align}
Recently, Qu, Wang, and Yao \cite{qu} proved \eqref{1,6mod2} and \eqref{1,10mod2} by finding the following general congruence:

If $j$ and $k$ are positive integers  with $1\le j\le k-1$ and $j$ odd, then for any nonnegative integer $n$,
$$a_{j,2k}(2n)\equiv 0~(\textup{mod}~2).$$
They also proved \eqref{1,6mod3} by finding the following identity:
$$\sum_{n=0}^\infty a_{1,6}(6n+3)q^n=3~\dfrac{E_2^3E_3^5}{E_1^6E_6},$$
which is analogous to Ramanujan's  \cite{rama-pn} so-called ``most" beautiful identity for the partition function $p(n)$, namely,
\begin{align}\label{rama-beautiful}
\sum_{n=0}^\infty p(5n+4)q^n=5~\dfrac{E_5^5}{E_1^6},
\end{align}
that immediately implies one of his three famous partition congruences, namely,
\begin{align*}
p(5n+4)&\equiv0~(\textup{mod}~5).
\end{align*}

Congruences in \eqref{1,3mod5}  were proved by Ding and Xia \cite{ding-xia}.

In this paper, we prove the remaining conjectural congruences  \eqref{1,10mod5} and  \eqref{3,10mod5} of Chan \cite{chan-actarith}.

We organize the paper in the following way. In Section \ref{sec3}, we find an exact generating function of $a(10n+9)$ analogous to \eqref{rama-beautiful} and deduce the congruences in \eqref{a50n} as well as  the new congruences in \eqref{a1250n}. In Section \ref {sec4}, we prove congruences \eqref{1,10mod5} and \eqref{3,10mod5}.

We employ Ramanujan's simple theta function identities and some other known identities for the Rogers-Ramanujan continued fraction, which is defined by
 \begin{equation*}
R(q):=\dfrac{q^{1/5}}{1}\+\dfrac{q}{1}\+\dfrac{q^2}{1}\+\dfrac{q^3}{1}\+\cds=q^{1/5}\dfrac{(q;q^5)_\infty(q^4;q^5)_\infty}{(q^2;q^5)_\infty(q^3;q^5)_\infty},~|q|<1.
\end{equation*}
In the next section, we present the background material on Ramanujan's theta functions and some lemmas that will be used in the subsequent sections.
\section{\textbf{Background material on Ramanujan's theta functions and some useful lemmas}}
Ramanujan's general theta function $f(a,b)$ is
defined by
\begin{equation*}
f(a,b):=\sum_{k=-\infty}^\infty a^{k(k+1)/2}b^{k(k-1)/2}, \quad |ab|<1.
\end{equation*}
We require the following special cases of $f(a,b)$:
\begin{align}\label{varphi}
\varphi(q)&:=f(q,q)=\sum_{j=-\infty}^\infty
q^{j^2}=(-q; q^2)^2_\infty(q^2; q^2)_\infty=\dfrac{E_2^5}
{E_1^2E_4^2},\\\intertext{and}
\label{psi}\psi(q)&:=f(q,q^3)=\sum_{j=0}^\infty q^{j(j+1)/2}=\dfrac{(q^2; q^2)_\infty}{(q; q^2)_\infty}=\dfrac{E_2^2}
{E_1},
\end{align}
where the product representations arise from Jacobi's famous triple product identity \cite[p. 35, Entry 19]{bcb3}
\begin{equation}\label{jtpi}
f(a,b)=(-a;ab)_\infty(-b;ab)_\infty(ab;ab)_\infty .
\end{equation}

Among many beautiful properties satisfied by $f(a,b)$ we recall the following from \cite{bcb3}.

\begin{lemma}\label{fabcd}\textup{(Berndt \cite[p. 45, Entry 29]{bcb3})}If $ab=cd$, then
\begin{align}\label{fabcdi}
f(a,b)f(c,d)&+f(-a,-b)f(-c,-d)=2f(ac,bd)f(ad, bc)\\\intertext{and}
\label{fabcdii}f(a,b)f(c,d)&-f(-a,-b)f(-c,-d)=2af\left(\dfrac{b}{c},ac^2d\right)f\left(\dfrac{b}{d}, acd^2\right).
\end{align}
\end{lemma}

In the next lemma, we state Jacobi's identity and an analogous identity.

\begin{lemma}\textup{(Berndt \cite[Theorem 1.3.9 and Corollary 1.3.22]{spirit})} We have
\begin{align}\label{jacobi}E_1^3=\sum_{k=0}^\infty(-1)^k(2k+1)q^{k(k+1)/2}\\\intertext{and}
\label{jacobi-type}\dfrac{E_1^2E_4^2}{E_2}=\sum_{n=-\infty}^\infty (3n+1)q^{3n^2+2n}.
\end{align}
\end{lemma}

In the following well-known results, the first two are 5-dissections of $E_1$ and $1/E_1$, respectively.
\begin{lemma}\textup{(Berndt \cite[p. 165]{spirit})} If ~ $T(q):=\dfrac{q^{1/5}}{R(q)}=\dfrac{(q^2;q^5)_\infty(q^3;q^5)_\infty}{(q;q^5)_\infty(q^4;q^5)_\infty}$, then
\begin{align}
\label{E1}E_1&=E_{25}\left(T(q^5)-q-\dfrac{q^{2}}{T(q^5)}\right)\\
\label{1byE1}\dfrac{1}{E_1}&= \dfrac{E_{25}^5}{E_5^6}\Big(T(q^5)^4+qT(q^5)^3+2q^2T(q^5)^2+3q^3T(q^5)+5q^4-\dfrac{3q^5}{T(q^5)}
\notag\\
&\quad+\dfrac{2q^6}{T(q^5)^2}-\dfrac{q^7}{T(q^5)^3}+\dfrac{q^8}{T(q^5)^4}\Big),\\\intertext{and}
\label{E1-6}11q&+\dfrac{E_1^6}{E_5^6}= T(q)^5-\dfrac{q^2}{T(q)^5}.
\end{align}
\end{lemma}

In the following two lemmas, we recall some useful results from our paper \cite{baruah-begum}.
\begin{lemma}\textup{(Baruah and Begum \cite[Lemma 1.3]{baruah-begum})}\label{xy} If $x=T(q)$ and $y=T(q^2)$, then
\begin{align}
\label{xy2}xy^2-\dfrac{q^2}{xy^2}&=K,\\
\label{x2-by-y}\dfrac{x^2}{y}-\dfrac{y}{x^2}&=\dfrac{4q}{K},\\
\label{y3-by-x}\dfrac{y^3}{x}+q^2\dfrac{x}{y^3}&=K+\dfrac{4q^2}{K}-2q,\\
\label{x3y}x^3y+\dfrac{q^2}{x^3y}&=K+\dfrac{4q^2}{K}+2q,
\end{align}
where $K=(E_2E_5^5)/(E_1E_{10}^5).$
\end{lemma}
\begin{lemma}\label{lemmaAB}\textup{(Baruah and Begum \cite[Eqs. (2.6), (2.7), (2.29)]{baruah-begum})}
\begin{align}\label{A-4qB}
\dfrac{E^5_5}{E_1^{4}E^{3}_{10}}&=\dfrac{E_5}{E_2^{2}E_{10}}+4q\dfrac{E^{2}_{10}}{E_1^{3}E_2},\\
\label{A-qB}
\dfrac{E^3_2E_5^2}{E_1^{2}E^2_{10}}&=\dfrac{E^5_5}{E_1E^{3}_{10}}
+q\dfrac{E^{2}_{10}}{E_2},\\
\label{A-5qB}
\dfrac{E^3_2E_5^2}{E_1^5E^{2}_{10}}&=\dfrac{E_5}{E_2^2E_{10}}+5q\dfrac{E^2_{10}}{E^3_1E_2}.
\end{align}
\end{lemma}

\section{\textbf{An exact generating function of $a(10n+9)$ and proofs of \eqref{a50n} and \eqref{a1250n}}}\label{sec3}
\begin{theorem} \label{th-a10n9} The generating function of $a(10n+9)$ is given by
\begin{align}\label{gen-a10n9}
\sum_{n=0}^\infty a(10n+9)q^n&=5\Big(46\dfrac{E_5E_{10}^2}{E_2^{2}} +460
q\dfrac{E_{10}^5}{E_1^3E_2}+1125q^2\dfrac{E_{10}^8}{E_1^6E_5} \notag\\
&\quad+1875q\dfrac{E_2^8E_5^{9}}{E_1^{16}}+15625q^2\dfrac{E_2^8E_5^{15}}{E_1^{22}}\Big).
\end{align}
\end{theorem}

Note that the congruence \eqref{a10n9} immediately follows from \eqref{gen-a10n9}.
\begin{pf1}
From \cite[Eq. (5.1)]{chan-actarith}, we have
$$\sum_{n=0}^\infty a(2n+1)q^n=\dfrac{E_2^8}{E_1^7},$$
which, with the aid of \eqref{A-qB} and \eqref{A-5qB}, may be simplified as
\begin{align}\label{a10nz}
\sum_{n=0}^\infty a(2n+1)q^n&=\dfrac{E_2^3E_{10}}{E_1^2E_5}+5q\dfrac{E_2^4E_{10}^4}{E_1^5E_5^2}\notag\\
&=\left(\dfrac{E_5^2}{E_1}+q\dfrac{E_{10}^5}{E_2E_5^3}\right)+5q\left(\dfrac{E_{10}^5}{E_2E_5^3}+5q\dfrac{E_{10}^8}{E_1^3E_5^4}\right)\notag\\
&=\dfrac{E_5^2}{E_1}+6q\dfrac{E_{10}^5}{E_2E_5^3}+25q^2\dfrac{E_{10}^8}{E_1^3E_5^4}.
\end{align}
Employing \eqref{1byE1} in the above, extracting the terms involving $q^{5n+4}$, dividing both sides of the resulting identity
by $q^4$, and then replacing $q^5$ by $q$, we find that
\begin{align}\label{a10nzz}
\sum_{n=0}^\infty a(10n+9)q^n&=5\dfrac{E_5^5}{E_1^4}+30q\dfrac{E_{10}^5}{E_1^3E_2}+1775q^2 \dfrac{E_2^8E_5^{15}}{E_1^{22}}+4425\dfrac{E_2^8E_5^{15}}{E_1^{22}}\left(T(q)^5-\dfrac{q^2}{T(q)^5}\right)\notag\\
&\quad+225\dfrac{E_2^8E_5^{15}}{E_1^{22}}\left(T(q)^{10}+\dfrac{q^4}{T(q)^{10}}\right).
\end{align}
By \eqref{E1-6}, the above can be simplified to
\begin{align*}
\sum_{n=0}^\infty a(10n+9)q^n&=5\Big(\dfrac{E_5^5}{E_1^4}+6q\dfrac{E_{10}^5}{E_1^3E_2} +45\dfrac{E_2^8E_5^{3}}{E_1^{10}}+1875q\dfrac{E_2^8E_5^{9}}{E_1^{16}}+15625q^2\dfrac{E_2^8E_5^{15}}{E_1^{22}}\Big).
\end{align*}
Employing \eqref{A-qB} in the above, we arrive at
\begin{align}\label{1-a10n9}
\sum_{n=0}^\infty a(10n+9)q^n&=5\Big(\dfrac{E_2^3E_5^2E_{10}}{E_1^5}+5q\dfrac{E_{10}^5}{E_1^3E_2} +45\dfrac{E_2^8E_5^{3}}{E_1^{10}}+1875q\dfrac{E_2^8E_5^{9}}{E_1^{16}}\notag\\
&\quad+15625q^2\dfrac{E_2^8E_5^{15}}{E_1^{22}}\Big).
\end{align}
With the aid of \eqref{A-5qB}, the above can be further simplified to
\begin{align}\label{2-a10n9}
\sum_{n=0}^\infty a(10n+9)q^n&=5\Big(45\left(\dfrac{E_2^3E_5^{2}E_{10}}{E_1^{5}}+5q\dfrac{E_2^4E_5E_{10}^4}{E_1^8}\right)+ \dfrac{E_2^3E_5^2E_{10}}{E_1^5}+5q\dfrac{E_{10}^5}{E_1^3E_2} \notag\\
&\quad+1875q\dfrac{E_2^8E_5^{9}}{E_1^{16}}+15625q^2\dfrac{E_2^8E_5^{15}}{E_1^{22}}\Big)\notag\\
&=5\Big(46\left(\dfrac{E_5E_{10}^2}{E_2^{2}} +5q\dfrac{E_{10}^5}{E_1^3E_2}\right) +5q\dfrac{E_{10}^5}{E_1^3E_2}+225q\left(\dfrac{E_{10}^5}{E_1^3E_2}+5q\dfrac{E_{10}^8}{E_1^6E_5}\right) \notag\\
&\quad+1875q\dfrac{E_2^8E_5^{9}}{E_1^{16}}+15625q^2\dfrac{E_2^8E_5^{15}}{E_1^{22}}\Big),
\end{align}
which is equivalent to \eqref{gen-a10n9}.
\end{pf1}

\begin{remark} If we extract the terms involving $q^{5n+r}$, $r=0,1,2,3$ after employing \eqref{1byE1} in \eqref{a10nz}, then we arrive at the generating functions of $a(10n+2r+1)$ as in \eqref{a10nzz}. But in these cases, the expressions involving $T(q)$ and $T(q^2)$  could not be expressed in terms of $E_1, E_2, E_5$ and $E_{10}$ as in the above case because of the non-availability of the expressions similar to those in \eqref{E1-6} -- \eqref{x3y}. Therefore, in Theorem \ref{th-a10n9}, we considered only the case $a(10n+9)$ among $a(10n+2r+1)$ where $0\leq r\leq 4$. 
\end{remark}

As corollaries to the above theorem, we now deduce the congruences in \eqref{a50n} originally conjectured by Chan in \cite{chan-actarith} and the new congruences in \eqref{a1250n}.

\begin{corollary} The congruences in \eqref{a50n} hold good.
\end{corollary}
\begin{proof}By the binomial theorem, we have
\begin{align}\label{binomial}E_1^5\equiv E_5 ~(\textup{mod}~5).\end{align}
Taking congruences modulo $5$ in \eqref{gen-a10n9} and using the above, we see that
\begin{align*}
\sum_{n=0}^\infty a(10n+9)q^n&\equiv5\times 46~E_5E_{10}E_2^3~(\textup{mod}~25),
\end{align*}
which can be rewritten with the aid of \eqref{jacobi} as
\begin{align}
\sum_{n=0}^\infty a(10n+9)q^n&\equiv5\times 46~E_5E_{10}\sum_{k=0}^\infty(-1)^k(2k+1)q^{k(k+1)}~(\textup{mod}~25).
\end{align}
As $k(k+1)\equiv0,~ 1,~ \textup{or}~2~(\textup{mod}~5)$, equating the coefficients of $q^{5n+r}$, $r=3,4$ from both sides of the above, we easily arrive at the last two congruences of \eqref{a50n}. Furthermore, we note that $k(k+1)\equiv 1~ (\textup{mod}~5)$ only when $k\equiv 2~ (\textup{mod}~5)$, that is, only when
$2k+1\equiv 0~ (\textup{mod}~5)$. Therefore, equating the coefficients of $q^{5n+1}$ from both sides of the above we arrive at the other congruence of \eqref{a50n}, to complete the proof. \end{proof}

\begin{corollary} The congruences in \eqref{a1250n} hold good.
\end{corollary}
\begin{proof}From \eqref{gen-a10n9}, we have
\begin{align}\label{a-cor2-1}
\sum_{n=0}^\infty a(10n+9)q^n&\equiv5\times 46 ~\left(\dfrac{E_5E_{10}^2}{E_2^{2}}+10q\dfrac{E_{10}^5}{E_1^{3}E_2}\right)~(\textup{mod}~125).
\end{align}

Now, let $\left[q^{5n+r}\right]\left\{F(q)\right\}$, $r=0,1,\ldots, 4$ denotes the terms after extracting the terms involving $q^{5n+r}$, dividing by $q^r$ and then replacing $q^5$ by $q$.

With the aid of \eqref{1byE1}, we have
\begin{align*}
\left[q^{5n+1}\right]\left\{\dfrac{E_5E_{10}^2}{E_2^{2}}\right\}=\dfrac{E_1E_{10}^{10}}{E_2^{10}}~\left(15q^3+10q\left(T(q^2)^5-\dfrac{q^4}{T(q^2)^5}\right)\right),
\end{align*}
which, by \eqref{E1-6}, implies
\begin{align}\label{a-cor2-2}
\left[q^{5n+1}\right]\left\{\dfrac{E_5E_{10}^2}{E_2^{2}}\right\}=10q\dfrac{E_1E_{10}^{4}}{E_2^{4}}+125q^3\dfrac{E_1E_{10}^{10}}{E_2^{10}}.
\end{align}

Again, by \eqref{A-4qB} and
\eqref{1byE1}, we obtain
\begin{align*}
\left[q^{5n+1}\right]\left\{10q\dfrac{E_{10}^5}{E_1^3E_2}\right\}&=\left[q^{5n+1}\right]\left\{\dfrac{5}{2}\left(\dfrac{E_{5}^5}{E_1^4}
-\dfrac{E_{5}E_{10}^2}{E_2^2}\right)\right\}\notag\\
&=5\Big(2\dfrac{E_{5}^{20}}{E_1^{19}}\left(x^{15}-\dfrac{q^{6}}{x^{15}}\right)+209q\dfrac{E_{5}^{20}}{E_1^{19}}\left(x^{10}+\dfrac{q^{4}}{x^{10}}\right)\notag\\
&\quad+5q\dfrac{E_1E_{10}^{10}}{E_2^{10}}\left(y^{5}-\dfrac{q^{4}}{y^{5}}\right) +920q^2\dfrac{E_{5}^{20}}{E_1^{19}}\left(x^{5}-\dfrac{q^{2}}{x^{5}}\right)\notag\\
&\quad+\dfrac{q^3}{2}\left(1015\dfrac{E_5^{20}}{E_1^{19}}-15\dfrac{E_1E_{10}^{10}}{E_2^{10}}\right)\Big).
\end{align*}
Employing \eqref{E1-6}, and then simplifying by using the identities in Lemma \ref{lemmaAB}, we find that
\begin{align}\label{a-cor2-3}
\left[q^{5n+1}\right]\left\{10q\dfrac{E_{10}^5}{E_1^3E_2}\right\}&=10\Big(\dfrac{E_1^{14}E_{10}^{3}}{E_2^{15}E_5}+150q
\dfrac{E_1^{11}E_{10}^{6}}{E_2^{14}E_5^2}+5650q^2\dfrac{E_1^{8}E_{10}^{9}}{E_2^{13}E_5^3}+101825q^3\dfrac{E_1^{5}E_{10}^{12}}{E_2^{12}E_5^4}\notag\\
&\quad+1068125q^4\dfrac{E_1^{2}E_{10}^{15}}{E_2^{11}E_5^5}+7042500q^5\dfrac{E_{10}^{18}}{E_1E_2^{10}E_5^6}\notag\\
&\quad+29800000q^6\dfrac{E_{10}^{21}}{E_1^{4}E_2^{9}E_5^7}+79000000q^7\dfrac{E_{10}^{24}}{E_1^{7}E_2^{8}E_5^8}\notag\\
&\quad+120000000q^8\dfrac{E_{10}^{27}}{E_1^{10}E_2^{7}E_5^9}+80000000q^9\dfrac{E_{10}^{30}}{E_1^{13}E_2^{6}E_5^{10}}\Big).
\end{align}
Invoking \eqref{a-cor2-2} and \eqref{a-cor2-3} in \eqref{a-cor2-1}, we obtain
\begin{align*}
\sum_{n=0}^\infty a(50n+19)q^n&\equiv5^2\times 92 ~\left(\dfrac{E_1^{14}E_{10}^{3}}{E_2^{15}E_5}+q\dfrac{E_1E_{10}^{4}}{E_2^{4}}\right)~(\textup{mod}~125),
\end{align*}
which, by \eqref{binomial} reduces to
\begin{align*}
\sum_{n=0}^\infty a(50n+19)q^n&\equiv5^2\times 92 ~\left(\dfrac{E_5^{2}}{E_1}+qE_1E_2E_{10}^{3}\right)~(\textup{mod}~125).
\end{align*}
Employing once again \eqref{1byE1} in the above, extracting the terms involving $q^{5n+4}$, dividing both sides by $q^4$, and then replacing $q^5$ by $q$, we arrive at
 \begin{align*}
\sum_{n=0}^\infty a(250n+219)q^n&\equiv5^2\times 92 ~E_5E_{10}E_2^{3}~(\textup{mod}~125).
\end{align*}
Employing \eqref{jacobi} in the above and then proceeding as in the proof of the previous corollary, we conclude that
 \begin{align*}
a(250(5n+r)+219)&\equiv0~(\textup{mod}~125), \quad \textup{for}~ r=1,3,4.
\end{align*}
Thus, we complete the proof  of \eqref{a1250n}.
\end{proof}

\begin{remark} Proceeding as in the proof of Theorem \ref{th-a10n9}, we may obtain the exact generating function of $a(50n+19)$, but the calculations and expressions are too lengthy and tedious even if we use \textbf{Mathematica}. Therefore, we decided not to include that lengthy generating function of $a(50n+19)$.
\end{remark}

\section{\textbf{Proofs of \eqref{1,10mod5} and \eqref{3,10mod5}}}\label{sec4}
In this section, we prove the congruences \eqref{1,10mod5} and \eqref{3,10mod5} originally conjectured by Chan \cite{chan-actarith}.

At first, setting $k=5$ and $j=1$ and $3$ in \cite[Eq. (2.7)]{qu}, we have
\begin{align*}4\sum_{n=0}^\infty a_{1,10}(n)q^n&=\dfrac{(-q;q^{10})_\infty^2(-q^9;q^{10})_\infty^2E_{10}^5}
{(q;q^{10})_\infty^2(q^9;q^{10})_\infty^2E_{20}^4}
-2\dfrac{E_{10}}{E_{20}^2}\sum_{n=-\infty}^\infty\dfrac{q^{5n(n+1)}}{1+q^{10n}}\\\intertext{and}
4\sum_{n=0}^\infty a_{3,10}(n)q^n&=\dfrac{(-q^3;q^{10})_\infty^2(-q^7;q^{10})_\infty^2E_{10}^5}
{(q^3;q^{10})_\infty^2(q^7;q^{10})_\infty^2E_{20}^4}
-2\dfrac{E_{10}}{E_{20}^2}\sum_{n=-\infty}^\infty\dfrac{q^{5n(n+1)}}{1+q^{10n}},
\end{align*}
respectively.

\noindent In view of \eqref{jtpi}, we can rewrite the above identities as
\begin{align}\label{gen-a1-10}4\sum_{n=0}^\infty a_{1,10}(n)q^n=\dfrac{f^2(q,q^9)E_{10}^5}
{f^2(-q,-q^9)E_{20}^4}
-2A(q^2)\\\intertext{and}
\label{gen-a3-10}4\sum_{n=0}^\infty a_{3,10}(n)q^n=\dfrac{f^2(q^3,q^7)E_{10}^5}
{f^2(-q^3, -q^{7})E_{20}^4}
-2A(q^2),
\end{align}
where $$A(q):=\dfrac{E_{5}}{E_{10}^2}\sum_{n=-\infty}^\infty\dfrac{q^{5n(n+1)/2}}{1+q^{5n}}.$$

\begin{pf-a1-10}
To prove \eqref{1,10mod5}, first we aim to find a  generating function of $a_{1,10}(2n+1)$. For that purpose, we need to find a 2-dissection of the first term of the right side of \eqref{gen-a1-10}. To that end, we recast \eqref{gen-a1-10} as
\begin{align}\label{gen-a1-10-1}4\sum_{n=0}^\infty a_{1,10}(n)q^n&=\dfrac{f^2(q,q^9)f^2(-q^3,-q^7)E_{10}^5}
{f^2(-q,-q^9)f^2(-q^3,-q^7)E_{20}^4}
-2A(q^2).
\end{align}

By Jacobi triple product identity, \eqref{jtpi}, we have
\begin{align}\label{1937}
f(-q,-q^9)f(-q^3,-q^7)&=\dfrac{(q;q^2)_\infty E_{10}^2}{(q^5;q^{10})_\infty}=\dfrac{E_1E_{10}^3}{E_2E_5},
\end{align}
and hence, \eqref{gen-a1-10-1} reduces to
\begin{align}\label{gen-a1-10-2}4\sum_{n=0}^\infty a_{1,10}(n)q^n&=\dfrac{E_2^2E_5^2}{E_1^2E_{10}E_{20}^4}f^2(q,q^9)f^2(-q^3,-q^7)
-2A(q^2)\notag\\
&=\dfrac{E_2}{E_{20}^4}\cdot \dfrac{\varphi(-q^5)}{\varphi(-q)}\left(f^2(q,q^9)f^2(-q^3,-q^7)\right)
-2A(q^2),
\end{align}
where \eqref{varphi} is used to arrive at the last equality.

Now, setting $a=q$, $b=q^9$, $c=-q^3$, and $d=-q^7$ in \eqref{fabcdi} and \eqref{fabcdii}, and then adding, we find that
\begin{align}\label{fabcdi3}
f(q,q^9)f(-q^3,-q^7)&=f(-q^4,-q^{16})f(-q^8,-q^{12})+qf(-q^6,-q^{14})f(-q^2,-q^{18}).
\end{align}

But, by Jacobi triple product identity, \eqref{jtpi}, we have
\begin{align*}
f(-q,-q^4)f(-q^2,-q^3)&=E_1E_5.
\end{align*}
Using the above identity  and \eqref{1937}  in \eqref{fabcdi3}, we see that
\begin{align*}
f(q,q^9)f(-q^3,-q^7)&=E_4E_{20}+q\dfrac{E_2E_{20}^3}{E_4E_{10}}.
\end{align*}
Therefore, \eqref{gen-a1-10-2} can be rewritten as
\begin{align}\label{gen-a1-10-3}4\sum_{n=0}^\infty a_{1,10}(n)q^n&=\dfrac{E_2}{E_{20}^4}\cdot \dfrac{\varphi(-q^5)}{\varphi(-q)}\left(E_4^2E_{20}^2+q^2\dfrac{E_2^2E_{20}^6}{E_4^2E_{10}^2}+2q\dfrac{E_2E_{20}^4}{E_{10}}\right)
-2A(q^2).
\end{align}
Replacing $q$ by $-q$ in the above, and then subtracting the resulting identity from \eqref{gen-a1-10-3}, we find that
\begin{align*}&4\sum_{n=0}^\infty a_{1,10}(n)q^n-4\sum_{n=0}^\infty a_{1,10}(n)(-q)^n\\
&=\dfrac{E_2}{E_{20}^4}
\left(E_4^2E_{20}^2+q^2\dfrac{E_2^2E_{20}^6}{E_4^2E_{10}^2}\right)\left(\dfrac{\varphi(-q^5)}{\varphi(-q)}-\dfrac{\varphi(q^5)}{\varphi(q)}\right)
+2q\dfrac{E_2^2}{E_{10}}\left(\dfrac{\varphi(-q^5)}{\varphi(-q)}+\dfrac{\varphi(q^5)}{\varphi(q)}\right).
\end{align*}
With the aid of the trivial identity $\varphi(q)\varphi(-q)=\varphi^2(-q^2)={E_2^4}/{E_4^2}$, we can rewrite the above as
\begin{align}\label{gen-a1-10-5}4&\sum_{n=0}^\infty a_{1,10}(n)q^n-4\sum_{n=0}^\infty a_{1,10}(n)(-q)^n\notag\\
&=\dfrac{E_4^2}{E_2^3E_{20}^4}
\left(E_4^2E_{20}^2+q^2\dfrac{E_2^2E_{20}^6}{E_4^2E_{10}^2}\right)\left(\varphi(q)\varphi(-q^5)-\varphi(-q)\varphi(q^5)\right)\notag\\
&\quad+2q\dfrac{E_4^2}{E_2^2E_{10}}\left(\varphi(q)\varphi(-q^5)+\varphi(-q)\varphi(q^5)\right).
\end{align}

Now, recall from \cite[p. 278]{bcb3} that
\begin{align*}\varphi(q)\varphi(-q^5)-\varphi(-q)\varphi(q^5)=4qE_4E_{20}.
\end{align*}

Furthermore, from Entries 25(i) and 25(ii) of \cite[p. 40]{bcb3}, it is easy to show that
\begin{align*}\varphi(q)\varphi(-q^5)+\varphi(-q)\varphi(q^5)=2\varphi(q^4)\varphi(q^{20})-8q^6\psi(q^8)\psi(q^{40}).
\end{align*}
Therefore, \eqref{gen-a1-10-5} becomes
\begin{align*}4&\sum_{n=0}^\infty a_{1,10}(n)q^n-4\sum_{n=0}^\infty a_{1,10}(n)(-q)^n\notag\\
&=4q\dfrac{E_4^3}{E_2^3E_{20}^3}
\left(E_4^2E_{20}^2+q^2\dfrac{E_2^2E_{20}^6}{E_4^2E_{10}^2}\right)+2q\dfrac{E_4^2}{E_2^2E_{10}}\left(2\varphi(q^4)\varphi(q^{20})-8q^6\psi(q^8)\psi(q^{40})\right).
\end{align*}
Extracting the terms involving $q^{2n+1}$ from both sides of the above, dividing by $q$, and then replacing $q^2$ by $q$, we find that
\begin{align}\label{gen-a1-10-55}&2\sum_{n=0}^\infty a_{1,10}(2n+1)q^n\notag\\
&=\dfrac{E_2^3}{E_1^3E_{10}^3}
\left(E_2^2E_{10}^2+q\dfrac{E_1^2E_{10}^6}{E_2^2E_{5}^2}\right)+\dfrac{E_2^2}{E_1^2E_{5}}\left(\varphi(q^2)\varphi(q^{10})-4q^3\psi(q^4)\psi(q^{20})\right)\notag\\
&=\dfrac{E_1^2}{E_5}+6q\dfrac{E_2E_{10}^3}{E_1E_{5}^2}+\dfrac{E_2^2}{E_1^2E_{5}}\left(\varphi(q^2)\varphi(q^{10})-4q^3\psi(q^4)\psi(q^{20})\right),
\end{align}
where we have used \eqref{A-5qB} to arrive at the last equality.

Now, to prove\eqref{1,10mod5}, we see from the above that it is enough to show that the coefficients of $q^{5n+2}$ of the terms  on the right side of the above are multiples of $5$. We accomplish this in the remaining part of the proof.

With the aid of \eqref{E1} and \eqref{1byE1}, we find that
\begin{align*}
&\left[q^{5n+2}\right]\left\{\dfrac{E_1^2}{E_5}+6q\dfrac{E_2E_{10}^3}{E_1E_{5}^2}\right\}\\
&=-\dfrac{E_5^2}{E_1}+6\dfrac{E_2^3E_5^5E_{10}}{E_1^8}
\left(\left(x^3y+\dfrac{q^2}{x^3y}\right)-2q\left(\dfrac{x^2}{y}-\dfrac{y}{x^2}\right)-5q\right),
\end{align*}
where $x$ and $y$ are as defined in Lemma \ref{xy}. Employing \eqref{x2-by-y} and \eqref{x3y} in the above, and then simplifying by using \eqref{A-4qB}, we obtain
\begin{align}\label{gen-a1-10-6}
&\left[q^{5n+2}\right]\left\{\dfrac{E_1^2}{E_5}+6q\dfrac{E_2E_{10}^3}{E_1E_{5}^2}\right\}=5\dfrac{E_5^2}{E_1}+30q\dfrac{E_2E_5E_{10}^3}{E_1^4}.
\end{align}

Next, by \eqref{binomial}, we have
\begin{align*}
\dfrac{E_2^2}{E_1^2E_{5}}~\varphi(q^2)\varphi(q^{10})&=\varphi(q^{10})\dfrac{E_4^5}{E_1^2E_{5}E_8^2}\\
&\equiv\dfrac{E_{20}\varphi(q^{10})}{E_5^2E_{40}}~E_1^3E_{8}^3~(\textup{mod}~5),
\end{align*}
which, with the aid of Jacobi's identity, \eqref{jacobi}, can be written as
\begin{align}\label{gen-a1-10-7}
&\dfrac{E_2^2}{E_1^2E_{5}}~\varphi(q^2)\varphi(q^{10})\notag\\
&\equiv\dfrac{E_{20}\varphi(q^{10})}{E_5^2E_{40}}~\sum_{j=0}^\infty\sum_{k=0}^\infty(-1)^{j+k}(2j+1)(2k+1)q^{j(j+1)/2+4k(k+1)}~(\textup{mod}~5).
\end{align}
We now observe that $j(j+1)/2+4k(k+1)\equiv2~(\textup{mod}~5)$ only when $j\equiv2~(\textup{mod}~5)$ and $k\equiv2~(\textup{mod}~5)$; i.e., only when
both $2j+1$ and $2k+1$ are multiples of 5.
Therefore, from \eqref{gen-a1-10-7}, we find that
\begin{align}\label{gen-a1-10-8}
&\left[q^{5n+2}\right]\left\{\dfrac{E_2^2}{E_1^2E_{5}}~\varphi(q^2)\varphi(q^{10})\right\}\equiv0~(\textup{mod}~5).
\end{align}

Finally, by \eqref{psi} and \eqref{binomial}, we have
\begin{align*}
q^3\dfrac{E_2^2}{E_1^2E_{5}}~\psi(q^4)\psi(q^{20})&=q^3\dfrac{E_2^2E_8^2}{E_1^2E_4E_5}~\psi(q^{20})\\
&\equiv\dfrac{\psi(q^{20})}{E_5^2}~E_1^3\cdot \dfrac{E_2^2E_{8}^2}{E_4}~(\textup{mod}~5),
\end{align*}
which can be rewritten, with the help of \eqref{jacobi} and \eqref{jacobi-type}, as
\begin{align}\label{gen-a1-10-9}
&q^3\dfrac{E_2^2}{E_1^2E_{5}}~\psi(q^4)\psi(q^{20})\notag\\
&\equiv\dfrac{\psi(q^{20})}{E_5^2}~\sum_{j=0}^\infty\sum_{k=-\infty}^\infty(-1)^{j}(2j+1)(3k+1)q^{j(j+1)/2+2k(3k+2)+3}~(\textup{mod}~5).
\end{align}
We observe that $j(j+1)/2+2k(3k+1)+3\equiv2~(\textup{mod}~5)$ only when $j\equiv2~(\textup{mod}~5)$ and $k\equiv3~(\textup{mod}~5)$; i.e., only when
both $2j+1$ and $3k+1$ are multiples of 5. Therefore, from \eqref{gen-a1-10-9}, we arrive at
\begin{align}\label{gen-a1-10-10}
&\left[q^{5n+2}\right]\left\{q^3\dfrac{E_2^2}{E_1^2E_{5}}~\psi(q^4)\psi(q^{20})\right\}\equiv0~(\textup{mod}~5).
\end{align}
With the aid of \eqref{gen-a1-10-6}, \eqref{gen-a1-10-8} and \eqref{gen-a1-10-10}, we conclude from \eqref{gen-a1-10-55} that
\begin{align*}a_{1,10}(10n+5)\equiv0~(\textup{mod}~5).
\end{align*}
Thus, we complete the proof of \eqref{1,10mod5}.
\end{pf-a1-10}
\begin{pf-a3-10}We can recast \eqref{gen-a3-10} as
\begin{align*}4\sum_{n=0}^\infty a_{3,10}(n)q^n=\dfrac{f^2(q^3,q^7)f^2(-q, -q^{9})E_{10}^5}
{f^2(-q^3, -q^{7})f^2(-q, -q^{9})E_{20}^4}
-2A(q^2),
\end{align*}
Proceeding exactly in the same way as in the proof of \eqref{1,10mod5}, we find that
\begin{align*}2\sum_{n=0}^\infty a_{3,10}(2n+1)q^n=\dfrac{E_1^2}{E_5}+6q\dfrac{E_2E_{10}^3}{E_1E_{5}^2}-\dfrac{E_2^2}{E_1^2E_{5}}\left(\varphi(q^2)\varphi(q^{10})-4q^3\psi(q^4)\psi(q^{20})\right).
\end{align*}
We notice that the right side of the above is almost the same as that of \eqref{gen-a1-10-55} except a negative sign. Hence, \eqref{3,10mod5} can be deduced as in the previous case.
\end{pf-a3-10}

\end{document}